\documentclass{amsart}
\usepackage{graphicx}
\usepackage{amsmath,bm}
\usepackage{amssymb, comment}
\usepackage[all]{xy}
\usepackage{color}

\newcommand{\Hom}{\operatorname{Hom}\nolimits}

\newcommand{\augcone}{\operatorname{augcone}\nolimits}

\newcommand{\pd}{\operatorname{pd}\nolimits}

\newcommand{\Ann}{\operatorname{Ann}\nolimits}

\newcommand{\Tor}{\operatorname{Tor}\nolimits}
\newcommand{\Ext}{\operatorname{Ext}\nolimits}

\newcommand{\n}{\operatorname{\mathfrak{n}}\nolimits}

\newcommand{\Z}{\operatorname{Z}\nolimits}

\newcommand{\V}{\operatorname{V}\nolimits}

\newcommand{\comments}[1]{}

\newtheorem{theorem}{Theorem}[section]

\newtheorem{corollary}[theorem]{Corollary}

\newtheorem{lemma}[theorem]{Lemma}
\newtheorem{proposition}[theorem]{Proposition}
\theoremstyle{definition}
\newtheorem{definition}[theorem]{Definition}

\theoremstyle{definition}

\theoremstyle{definition}

\theoremstyle{definition}
\newtheorem*{example}{Example}
\theoremstyle{definition}

\theoremstyle{definition}

\theoremstyle{remark}
\newtheorem*{remark}{Remark}

\theoremstyle{definition}

\begin{document}

\title{Support varieties over complete intersections made easy}
\author{Petter Andreas Bergh \& David A.\ Jorgensen}
\address{Petter Andreas Bergh \\ Institutt for matematiske fag \\
  NTNU \\ N-7491 Trondheim \\ Norway}
\email{bergh@math.ntnu.no}
\address{David A.\ Jorgensen \\ Department of mathematics \\ University
of Texas at Arlington \\ Arlington \\ TX 76019 \\ USA}
\email{djorgens@uta.edu}


\begin{abstract}
Using an alternate description of support varieties of pairs of modules over a complete intersection, we give several new applications of such varieties, including results for support varieties of intermediate complete intersections.  Simpler proofs of known results are also given.
\end{abstract}



\thanks{Part of this work was done while we were visiting the Mittag-Leffler Institute in February and March 2015. We would like to thank the organizers of the Representation Theory program.}

\maketitle

\section{Introduction}\label{sec:intro}
The point of this paper is to exploit an alternate description of support varieties for cohomology of pairs of modules over a complete intersection.  A coordinatized version of this alternate description was first given in \cite{AvramovBuchweitz}, and subsequently generalized to non-complete intersections in \cite{Jorgensen1}.  However, this alternate description has not really been explored; we find that it leads to conceptually clearer proofs of known results, as well as new results for support varieties over related complete intersections, and related support varieties.   

The paper is organized as follows.  In Section 2 we give the alternate definition of support variety, and use it to reprove known results. In Section 3 we prove new results for support varieties of related complete intersections, which we call intermediate complete intersections.  We give a new proof in Section 4 of realizability of support varieties over complete intersections, namely, that every cone in affine space occurs as the support variety of some finitely generated module, even one of finite length. In Section 5 we give further applications of support varieties. Finally, in the appendix we recount the usual definition of support varieties for cohomology over complete intersections, as well as other facts used in the main part of the paper.

\section{Support varieties made easy}\label{sec:easy}

Throughout the paper we fix a regular local ring $(Q, \n, k)$ and an ideal $I$ generated by a regular sequence of length $c$
contained in $\n^2$. By \ref{appendix:algebraically closed} below we can assume that the residue field $k$ is algebraically closed. 
We denote by $R$ the complete intersection ring 
$$R = Q/I,$$
and by $V$ the $k$-vector space 
$$V=I/\n I.$$
For an element $f\in I$, we let $\overline f$ denote its image in $V$.

We now give our working definition of support variety.  A coordinatized version of it was originally introduced in \cite[Theorem 2.5]{AvramovBuchweitz} as an alternate description of support variety (see also \cite[Corollary (3.11)]{Avramov}.  This alternate description was subsequently generalized to non-complete intersections in \cite{Jorgensen1}.  The fact that our definition is coordinate-free leads to conceptually easier proofs of known results on support varieties, as well as to several new applications.

\begin{definition}
Let $M$ and $N$ be finitely generated $R$-modules.  We define the \emph{support variety} of the pair $(M,N)$ to be 
\[ 
\V_R(M,N)=\{ \overline f\in V\mid \Ext_{Q/(f)}^i(M,N)\ne 0 \text{ for infinitely many } i\} \cup  \{ 0 \}.
\]
We set $\V(M)=\V(M,k)$.
\end{definition}

\begin{remark}
There are two things that are admittedly unclear regarding this definition:
\begin{enumerate}
\item That the definition makes sense, in that if $f,g\in I$ are such that 
$\overline f =\overline g$ in $V$, then $\Ext_{Q/(f)}^i(M,N)=0$ for all $i\gg 0$ if and only if
$\Ext_{Q/(g)}^i(M,N)=0$ for all $i \gg 0$.
\item That $\V_R(M,N)$ is actually a closed set in $V$; there is no mention of coordinates in the definition.
\end{enumerate}
In the appendix (see \ref{cohomology operators}) we recall the usual definition of support variety of pairs of modules over a complete intersection.  Choosing a minimal generating set $f_1,\dots,f_c$ of $I$, and thus coordinates of $V$, Avramov and Buchweitz prove in \cite[Theorem 2.5]{AvramovBuchweitz} that these two notions of support variety agree under the identification $V \ni \overline f_i \leftrightarrow e_i\in k^c$.  This answers (2).  That (1) holds is implicit in the course of the proof of \emph{loc.\ cit.} We mention also that, after coordinatizing $V\ni \overline f_i \leftrightarrow e_i\in k^c$, both properties (1) and (2) are proved explicitly to hold more generally for finitely generated modules over arbitrary quotients of integral domains in \cite[Theorem 2.1]{Jorgensen1}.
\end{remark}

We first reiterate some basic properties of support varieties.  Properties (1)-(7) in the following result can be found in 
\cite[Theorem 5.6]{AvramovBuchweitz}, and (8) follows from 
\cite[Existence Theorem 7.4(4)]{AvramovIyengar}. Due to our definition, the proofs we give are significantly simpler than those in the existing literature.  

\begin{theorem} \label{thm:props}
The following hold for finitely generated $R$-modules $M$ and $N$.
\begin{enumerate}
\item $\V_R(k)=V$. 
\item If $\Ext^i_R(M,N)=0$ for all $i \gg 0$, then $\V_R(M,N)=\{0\}$. In particular, if $M$ has finite projective dimension over $R$, then $\V_R(M,N)=\{0\}$ for all $R$-modules $N$.
\item $\V_R(M,N)=\V_R(M)\cap\V_R(N)$.
\item $\V_R(M,M)=\V_R(k,M)=\V_R(M)$.
\item If $M'$ is a syzygy of $M$ and $N'$ is a syzygy of $N$, then $\V_R(M,N)=\V_R(M',N')$. \label{syzygy}
\item If $0\to M_1\to M_2\to M_3\to 0$ and $0\to N_1\to N_2\to N_3\to 0$ are short exact sequence of finitely generated $R$-modules, then for $\{h,i,j\}=\{1,2,3\}$ there are inclusions
\[
\V_R(M_h,N)\subseteq\V_R(M_i,N)\cup\V_R(M_j,N);
\]
\[
\V_R(M,N_h)\subseteq\V_R(M,N_i)\cup\V_R(M,N_j).
\]
\item If $M$ is Cohen-Macaulay of codimension  $m$, then 
\[
\V_R(M)=\V_R(\Ext_R^m(M,R)).
\]
In particular, if $M$ is a maximal Cohen-Macaulay $R$-module, then
$\V_R(M)=\V_R(\Hom_R(M,R))$.
\item \label{regseq} If $x_1,\dots,x_d$ is an $M$-regular sequence, then 
\[
\V_R(M)=\V_R(M/(x_1,\dots,x_d)M).
\]
\end{enumerate}
\end{theorem}

\begin{proof}
(1):  This follows since $k$ has infinite projective dimension over every singular quotient ring of $Q$, in particular for every $Q/(f)$, when $0 \neq \overline f \in V$.

(2):  By the change of rings long exact sequence of Ext (see for example \cite[2.1]{Jorgensen1}) it is well-known that if $A$ is a local ring and $B=A/(x)$, where $x$ is a non-zerodivisor of $A$, then vanishing of all higher Ext of a pair of $B$-modules over $A$ holds provided it does so over $B$.  
Inductively, the same is true for the quotient by a regular sequence. Thus, if $\Ext_R^i(M,N)=0$ for all $i\gg 0$, then $\Ext_{Q/(f)}^i(M,N)=0$ for all $i\gg 0$, and for every non-zero $\overline f$ in $V$. 

(3):  Suppose that $\overline f$ is not in $\V(M,N)$.  Then $\Ext^i_{Q/(f)}(M,N)=0$ for all $i\gg 0$.  By taking syzygies of $M$ and $N$ over $Q/(f)$, we may assume that $M$ and $N$ are maximal Cohen-Macaulay $Q/(f)$-modules.  Then we have the conversion isomorphisms
$\Ext^i_{Q/(f)}(\Omega_{Q/(f)}^{-t}(M),N)\cong\Tor_{t-i-1}^{Q/(f)}(M^*,N)$ for $t\ge 3$ and $1\le i\le t-2$, where $\Omega_{Q/(f)}^{-t}(M)$  denotes the $t$th cosyzygy of $M$ 
\cite[Lemma 1.1]{HunekeJorgensen}.  Choosing $t$ large enough, and noting that $\Ext^i_{Q/(f)}(\Omega_{Q/(f)}^{-t}(M),N)=0$ for all $i>0$, we realize two consecutive $\Tor_i^{Q/(f)}(M^*,N)$ vanishing for $i>0$.  Thus by \cite[(1.9)]{HunekeWiegand}, either $M^*$ or $N$ is free, and so either $M$ or $N$ is free.
This means for the original $M$ and $N$ that either $\overline f$ is not in $\V(M)$ or $\overline f$ is not in $\V(N)$, in other words, $\overline f$ is not in $\V(M)\cap\V(N)$, and we have shown
$\V(M)\cap\V(N)\subseteq \V(M,N)$.  The reverse containment is an easy exercise, using (2).

(4):  This statement follows immediately from (3) and (1).

(5):  Suppose that $M'=\Omega_R^d(M)$ and $N'=\Omega_R^e(N)$.  Then we have
\[
\Ext_{Q/(f)}^i(M,N)\cong\Ext_{Q/(f)}^{i-d+e}(\Omega_{Q/(f)}^d(M),\Omega_{Q/(f)}^e(N))=\Ext_{Q/(f)}^{i-d+e}(M',N')
\] 
for all $i\gg 0$.  Thus the first Ext vanishes for all $i\gg 0$ if and only if so does the third.  The result follows.

(6): The long exact sequence of Ext over $Q/(f)$ corresponding to the first short exact sequence is
\[
\cdots \to \Ext^n_{Q/(f)}(M_3,N) \to \Ext^n_{Q/(f)}(M_2,N)\to \Ext^n_{Q/(f)}(M_1,N) \to \cdots
\]
If $\overline f$ is not in $\V(M_i,N)\cup\V(M_j,N)$ for $\{h,i,j\}=\{1,2,3\}$, then $\overline f$ is not in $\V(M_i,N)$ and $\overline f$ is not in $\V(M_j,N)$.  Thus we have
$\Ext^n_{Q/(f)}(M_i,N)=\Ext^n_{Q/(f)}(M_j,N)=0$ for all $n \gg 0$.  Therefore $\Ext^n_{Q/(f)}(M_h,N)=0$ for all $n \gg 0$, and so $\overline f$ is not in $\V(M_h,N)$.  We have shown
$\V(M_h,N)\subseteq\V(M_i,N)\cup\V(M_j,N)$.  The other containment is proved similarly.

(7) If $M$ is Cohen-Macaulay of codimension $m$ over $R$, then it is Cohen-Macaulay of codimension $m+c-1$ over $Q/(f)$.  If moreover $M$ has finite projective dimension over $Q/(f)$, then it is a perfect $Q/(f)$-module of grade $m+c-1$.  It follows that $\Ext^{m+c-1}_{Q/(f)}(M,Q/(f))$ is also a perfect $Q/(f)$-module of grade $m+c-1$ \cite[Exercise 1.4.26]{BrunsHerzog}.  By Rees' Theorem we have 
$\Ext^m_R(M,R)\cong\Ext^{m+c-1}_{Q/(f)}(M,Q/(f))$.  Thus $M$ and $\Ext^m_R(M,R)$ have finite projective dimension over $Q/(f)$ simultaneously. 

(8)  By induction it suffices to consider a single non-zerodivisor $x$.  We have the short exact sequence of $R$-modules $0 \to M \xrightarrow{x} M \to M/xM \to 0$.  The corresponding long exact sequence of Ext over $Q/(f)$ is
\[
\cdots \to \Ext^i_{Q/(f)}(M,k) \xrightarrow{x} \Ext^i_{Q/(f)}(M,k) \to \Ext^{i+1}_{Q/(f)}(M/xM,k) \to\cdots
\]
We have $\overline f$ is not in $\V_R(M)$ if and only if $\Ext^i_{Q/(f)}(M,k)=0$ for all $i\gg 0$. By Nakayama's Lemma this is equivalent to $\Ext^i_{Q/(f)}(M/xM,k)=0$ for all $i\gg 0$, that is,  $\overline f$ is not in $\V_R(M/xM)$. 
\end{proof}

\section{Intermediate complete intersections and support varieties}\label{sec:intermediate}

Our definition of support variety makes it possible to relate varieties over related complete intersections.

\subsection*{Support varieties of intermediate complete intersections}
Any complete intersection $R=Q/I$ of codimension at least 2 has a myriad of related complete intersections. Namely, if $W$ is a subspace of $V$, then choosing preimages in $I$ of a basis of $W$ we obtain another regular sequence \cite[Theorem 2.1.2(c,d)]{BrunsHerzog}, and the ideal $J\subseteq I$ it generates.  We thus get natural projections of complete intersections $Q\to Q/J\to Q/I=R$. We call $Q/J$ a \emph{complete intersection intermediate to $Q$ and $R$}, or when the context is clear, simply an \emph{intermediate complete intersection}. 

Consider the natural map of $k$-vector spaces
\[
\varphi_{J}:J/\n J\to W\subseteq V
\] 
defined by $f+\n J \mapsto f+\n I$.  This is an isomorphism: it is onto by construction,  and one-to-one since $J\cap \n I=\n J$. A basic question is how the support variety of a pair of $R$-modules $(M,N)$ over $R$ relates to that of the pair $(M,N)$ over the intermediate complete intersection $R' = Q/J$.  Using our definition of support variety, together with the map just defined, we get a very simple and appealing answer to this question. 

\begin{theorem}\label{thm:main}
Let $W$ be a subspace of $V$.  Consider the intermediate complete intersection $R'=Q/J$, where
$J$ is an ideal generated by a regular sequence consisting of preimages in $I$ of a basis of $W$.
Then for all finitely generated $R$-modules $M$ and $N$ we have
\[
\varphi_{J}(\V_{R'}(M,N)) = \V_R(M,N) \cap W
\]
\end{theorem}

\begin{proof} 
Let $f+\n I\in\varphi_J(\V_{R'}(M,N))$, where $f\in J$.  Then $f+\n J$
is in $\V_{R'}(M,N)$, which means that $\Ext^i_{Q/(f)}(M,N)$ is nonzero for infinitely many $i$.
But then $f+\n I$ is in $\V_R(M,N)\cap W$.  The reverse inclusion is equally tautological.
\end{proof}

\subsection*{Equivalent intermediate complete intersections}

Let $Q/J$ and $Q/J'$ be two complete intersections intermediate to $Q$ and $R$.  The condition that  
\[
\varphi_J(J/\n J)=\varphi_{J'}(J'/\n J')
\] 
in $V$ defines an equivalence relation on the set of such intermediate complete intersections.  The next result states that equivalent intermediate complete intersections have the same support varieties.

\begin{proposition}\label{equivalentcompleteintersections}
Suppose that $R'=Q/J$ and $R''=Q/J'$ are equivalent complete intersections intermediate to $Q$ and $R$, that is, $\varphi_J(J/\n J)=\varphi_{J'}(J'/\n J')$.  Then for all finitely generated $R$-modules $M$ and $N$ we have
\[
\varphi_J(V_{R'}(M,N))=\varphi_{J'}(V_{R''}(M,N))
\]
\end{proposition}

\begin{proof} This just follows from Theorem \ref{thm:main}: 
\begin{align*}
\varphi_J(V_{R'}(M,N))&=\V_R(M,N)\cap \varphi_J(J/\n J)\\
&=\V_R(M,N)\cap \varphi_{J'}(J'/\n J')\\
&=\varphi_{J'}(V_{R''}(M,N)).
\end{align*}
\end{proof}

\section{realizability}

Recall that  $V=I/\n I$.  For the following results, we associate to $V$ the coordinate ring $k[x_1,\dots,x_c]$, coordinatized by $x_i \leftrightarrow \overline f_i\in V$, where $f_i\in Q$, $1\le i\le c$, form a minimal generating set for the ideal $I$. For a polynomial $p\in k[x_1,\dots,x_c]$, we let $\Z(p)$ denote the zero set of $p$ in $V$.

\begin{lemma}\label{lemma:realize}  Let $p$ be a homogeneous polynomial in $k[x_1,\dots,x_c]$, the coordinate ring of $V$, and $M$ a finitely generated $R$-module.  Then there exists a finitely generated $R$-module $K_p$ such that 
\[
\V_R(K_p)=\V_R(M)\cap \Z(p)
\] 
\end{lemma}

\begin{proof} Let $f\in I$ be an element such that $\overline f$ is nonzero in $V$.  Then we have the natural projection $Q/(f)\to R$, which induces natural maps $\varphi:\Ext_R(M,k) \to \Ext_{Q/(f)}(M,k)$.
Consider the map $\xi_k:k[\chi_1,\dots,\chi_c] \to \Ext_R(k,k)$ from the polynomial ring of cohomology operators associated to $f_1,\dots,f_c$ (see \ref{cohomology operators}), and the map $\xi'_M:k[\chi] \to \Ext_{Q/(f)}(k,k)$, from the polynomial ring of the cohomology operator associated to $f$. 
Write $f=a_1f_1+ \cdots + a_cf_c$.
Then by \cite[1.1.1]{AvramovBuchweitz} (which follows from \cite[Proposition 1.7]{Eisenbud}) we have commutativity of the squares
\[
\xymatrixrowsep{2pc}
\xymatrixcolsep{3pc}
\xymatrix
{
\Ext_R^i(M,k) \ar[r]^{\xi_k(\chi_j)\cdot}\ar[d]^{\varphi_i} & \Ext_R^{i+2}(M,k) \ar[d]^{\varphi_{i+2}} \\
 \Ext_{Q/(f)}^i(M,k) \ar[r]^{a_j\xi'_k(\chi)\cdot} & \Ext_{Q/(f)}^{i+2}(M,k)
}
\]
for all $i\ge 0$ and all $1\le j \le c$.  In other words, the equality
\[
\varphi\circ\xi_k(\chi_j)= a_j\xi'_k(\chi)\circ\varphi
\] 
holds in each degree and for all $1\le j \le c$. A straightforward computation yields
\[
\varphi\circ\xi_k(p(\chi_1,\dots,\chi_c))=p(a_1,\dots,a_c)\xi'_k(\chi)^d\circ\varphi
\]
in each degree, where $d$ is the degree of $p$.  We note that since $Q/(f)$ is a hypersurface, if $\Ext^i_{Q/(f)}(M,k)\ne 0$ for all $i\gg 0$, then $\xi_k'(\chi)$ is an isomorphism in all sufficiently high degrees (see \ref{appendix:hypersurface} in the appendix).  Consequently, so is $\xi_k'(\chi)^d$ under the same condition. 

Consider the short exact sequence of $R$-modules associated to $p(\chi_1,\dots,\chi_c)$ (see \ref{appendix:ses} in the appendix):
\[
0 \to M \to K_p \to \Omega_R^{d-1}M \to 0
\]
This gives rise to the commutative diagram for $i \gg 0$
\[
\xymatrixrowsep{2pc}
\xymatrixcolsep{1.9pc}
\xymatrix{
\Ext_R^{i+d-1}(M,k) \ar[r]\ar[d]^{\varphi_{i+d-1}} & \Ext_R^i(K_p,k) \ar[r]\ar[d] & \Ext_R^i(M,k) \ar[r]^{\alpha_i}\ar[d]^{\varphi_i} & \Ext_R^{i+d}(M,k) \ar[d]^{\varphi_{i+d}}  \\
\Ext_{Q/(f)}^{i+d-1}(M,k) \ar[r] & \Ext_{Q/(f)}^i(K_p,k) \ar[r] & \Ext_{Q/(f)}^i(M,k) \ar[r]^{\beta_i} & \Ext_{Q/(f)}^{i+d}(M,k)
}
\]
where $\alpha_i=\xi_k(p(\chi_1,\dots,\chi_c))\cdot$, $\beta_i=p(a_1,\dots,a_c)\xi'_k(\chi)^d\cdot$, and we have used the fact that 
$\Ext^i_{Q/(f)}(\Omega^{d-1}_R(M),k)\cong \Ext^i_{Q/(f)}(\Omega_{Q/(f)}^{d-1}(M),k)$ for all $i\gg 0$
(see \ref{appendix:syzygies} in the appendix).

Now we are ready to prove the lemma.  Suppose that $\overline f\in V$ is not in $\V_R(K_p)$.  Then $\Ext_{Q/(f)}(K_p,k)=0$ for all $i\gg 0$.  In this case, either $\Ext_{Q/(f)}(M,k)=0$ for all $i \gg 0$, in which case $\overline f$ is not in $\V_R(M)$, or the maps $\beta_i$ in the diagram above are isomorphisms for all $i\gg 0$.  The latter happens only if $p(a_1,\dots,a_c)\ne 0$, in other words,  $\overline f$ is not in $\Z(p)$. Thus $\overline f$ is not in $\V_R(M)\cap\Z(p)$ and we have 
$\V_R(M)\cap \Z(p)\subseteq\V_R(K_p)$.

Now suppose that $\overline f$ is not in $\V_R(M)\cap \Z(p)$.  If $\overline f$ is not in 
$\V_R(M)$, then $\Ext^i_{Q/(f)}(M,k)=0$ for all $i\gg 0$,  and this forces $\Ext^i_{Q/(f)}(K_p,k)=0$ for all $i\gg 0$, so that $\overline f$ is not in $\V_R(K_p)$.  So assume $\overline f$ is in $\V_R(M)$ and not in $\Z(p)$.  Then $p(a_1,\dots,a_c)\ne 0$ and the maps $\beta_i$ are isomorphisms for all $i\gg 0$.  Thus $\Ext^i_{Q/(f)}(K_p,k)=0$ for all $i\gg 0$ and so $\overline f$ is not in $\V_R(K_p)$.  This finishes the proof.
\end{proof}

\begin{theorem}\label{thm:realize} Let $\mathfrak C$ be any cone in $V$.  Then there exists a finitely generated $R$-module $M$ such that
\[
\V_R(M)=\mathfrak C
\]
\end{theorem}

\begin{proof} Since $\mathfrak C$ is a cone in $V$, there exist homogeneous polynomials $p_1,\dots,p_n$ in $k[x_1,\dots,x_c]$ such that $\mathfrak C =\Z(p_1,\dots,p_n)$.  
We induct on $n$.  

For the base case $n=1$, we apply Lemma \ref{lemma:realize} to the $R$-module $M=k$.  Thus there exists a finitely generated $R$-module $K_1$ such that
\[
\V_R(K_1)=\V_R(k)\cap\Z(p_1)=V\cap\Z(p_1)=\Z(p_1)
\]

Now suppose that we have constructed a finitely generated $R$-module $K_{n-1}$ such that
$\V_R(K_{n-1})=\Z(p_1,\dots,p_{n-1})$.  Applying Lemma \ref{lemma:realize} to $K_{n-1}$ now, we 
obtain a finitely generated $R$-module $K_n$ such that
\[
\V_R(K_n)=\V_R(K_{n-1})\cap\Z(p_n)=\Z(p_1,\dots,p_{n-1})\cap\Z(p_n)=\Z(p_1,\dots,p_n)
\]
\end{proof}

\begin{remark}
In fact every cone in $V$ is realized by an $R$-module of finite length. Indeed, suppose that 
$\mathfrak C$ is a cone in $V$.  Then from Theorem \ref{thm:realize} there exists an $R$-module $M$ such that $\V_R(M)=\mathfrak C$.  From (\ref{syzygy}) of Theorem \ref{thm:props} we can assume that $M$ is maximal Cohen-Macaulay.  Letting $\bm x$ denote a maximal regular sequence on $R$, we have that it is also a regular sequence on $M$.  Thus by (\ref{regseq}) of Theorem \ref{thm:props} we have that $\mathfrak C=\V_R(M)=\V_R(M/(\bm x)M)$, and $M/(\bm x)M$ has finite length.
\end{remark}

\section{Applications}

The following theorem shows that support varieties for complementary complete intersections can be realized by a \emph{single} module.
 
\begin{theorem}\label{thm:tensor}
Suppose that $V=W_1\oplus\cdots\oplus W_r$, for subspaces $W_i$, $1\le i\le r$, of $V$.  Choose complete intersections $R_i=Q/J_i$ intermediate to $Q$ and $R$ such that $\varphi_{Ji}(J_i/\n J_i)=W_i$, $i\le i\le r$.  Let $M_i$ be a maximal Cohen-Macaulay $R_i$-module for $1\le i\le r$. Then for all $1\le i\le r$ we have
\[
\varphi_{Ji}(\V_{R_i}(M_i))=\V_R(M_1\otimes_Q \cdots \otimes_QM_r)\cap W_i
\]
\end{theorem}

\begin{proof} By induction it suffices to prove the theorem when $r=2$.  By symmetry it suffices to prove that $\varphi_{J_1}(\V_{R_1}(M_1))=\V_R(M_1 \otimes_QM_2)\cap W_1$, and by Theorem \ref{thm:main}, for this it suffices to prove that $\V_{R_1}(M_1)=\V_{R_1}(M_1\otimes_Q M_2)$.
Choose $\overline f$ in $W_1$.  Then $f$ is regular on $R_2$. Since $M_2$ is a maximal Cohen-Macaulay $R_2$-module, we have that $f$ is also regular on $M_2$.  Thus a (necessarily) finite free resolution 
$F$ of $M_2$ over $Q$ yields the finite free resolution 
$Q/(f)\otimes_QF$ of $Q/(f)\otimes_QM_2$ over $Q/(f)$.  
Since $f$ annihilates $M_1$, we have that $M_1$ is a 
$Q/(f)$-module, and we have the standard isomorphism 
$\Tor^{Q/(f)}_i(M_1,Q/(f)\otimes_Q M_2)\cong \Tor^Q_i(M_1,M_2)$ for all $i>0$.  Now by 
\cite[2.2]{Jorgensen2} we have $\Tor^Q_i(M_1,M_2)=0$ for all $i>0$.  Thus a free resolution of 
$M_1\otimes_Q M_2 \cong M_1\otimes_Q (Q/(f)\otimes_Q M_2)$ over $Q/(f)$ is gotten by tensorring
a free resolution of $M_1$ over $Q/(f)$ with a free resolution of $Q/(f)\otimes_Q M_2$ over $Q/(f)$.  Since the latter resolution is finite, $M_1\otimes_Q M_2$ will have a finite free resolution over $Q/(f)$ if and only if $M_1$ has one, and this is what we needed to show. 
\end{proof}

One may ask how support varieties change upon moving to higher codimension.  As a corollary of Theorem 
\ref{thm:tensor}, we give one basic case where it does not change.

\begin{corollary}Suppose that $W$ and $W'$ are subspaces of $V$ with $V=W\oplus W'$. Let $R'=Q/J$ 
and $R''=Q/J'$ be intermediate complete intersections corresponding to $W$ and $W'$, so that 
$\varphi_J(J/\n J)=W$ and $\varphi_{J'}(J'/\n J')=W'$. Let $M$ be a maximal Cohen-Macaulay 
$R'$-module.  Then $\V_R(M/J'M) \subseteq W$, and we have
\[
\varphi_{J}(\V_{R'}(M))=\V_R(M/J'M)
\]
\end{corollary}

\begin{proof}We first establish the inclusion.  Since $\V_{R''}(R'')=0$, by Theorem \ref{thm:tensor} we have $0=\varphi_{J'}(\V_{R''}(R''))=\V_R(M\otimes_QR'')\cap W'$.  Thus $\V_R(M\otimes_QR'')\subseteq W$, and since $M\otimes_QR''\cong M/J'M$, we have the desired inclusion.  Noting that $J'$ is generated by an $M$-regular sequence, and by using (\ref{regseq}) of Theorem \ref{thm:props} together with Theorem \ref{thm:tensor}, we now have
$\varphi_{J}(\V_{R'}(M))=\varphi_{J}(\V_{R'}(M/J'M))=\V_R(M/J'M)\cap W=\V_R(M/J'M)$. 
\end{proof}

The following theorem gives a useful criterion for determining when syzygies of modules over complete intersections split (see, for example \cite{O'CarrollPopescu}).  It uses a result of the first author, \cite[Theorem 3.1]{Bergh}, which says that if $M$ is a maximal Cohen-Macaulay $R$-module, then $\V_R(M)$ is irreducible if $M$ is indecomposable.

\begin{theorem}
Let $M$ be a maximal Cohen-Macaulay $R$-module, and $W$ a subspace of $V$.  Let $R'$ be an intermediate complete intersection corresponding to $W$. Suppose that 
$\V_R(M)$ is irreducible and  $\V_R(M)\cap W$ is reducible.  Then $M$ is an indecomposable 
$R$-module, but the $R'$-syzygies of $M$ split.
\end{theorem}

\begin{proof}
By Theorem \ref{thm:main} we have that $\V_{R'}(M)=\V_R(M)\cap W$ decomposes as a union of non-zero closed sets with zero intersection.  Thus the same is true for any syzygy of $M$ over $R'$.  Choose a syzygy $M'$ of $M$ over $R'$ such that $M'$ is a maximal Cohen-Macaulay $R'$-module.  Then by \cite[Theorem 3.1]{Bergh},  the module $M'$ splits as a direct sum of nonzero submodules.
\end{proof}

\begin{example}
Consider $Q=k[[x,y,z]]$, and $R=Q/(x^2,y^2,z^2)$.  Then $R$ is a codimension 3 complete intersection.  We know there exists an $R$-module $M$ such that $\V_R(M)=\Z(x_1x_2-x_3^2)$.
We actually construct said module.  One can show that $k$ has a minimal free resolution $F$ over $R$ of the 
form 
\[
\cdots \to R^{21} \to R^{15} \xrightarrow{\partial_4} R^{10} \to R^6 \to R^3 \xrightarrow{\partial_1} R \to k \to 0
\]  
where $\xi_k(\chi_1\chi_2-\chi_3^2)\in\Ext_R^4(k,k)$ corresponds to the map $R^{15}\to R$ that sends $e_1$ to $-1$, $e_{13}$ to 1, and $e_i$ to 0 for $i\ne 1,13$, where the $e_i$ denote a standard basis of $R^{15}$.  With respect to the standard bases, the differentials $\partial_1$ and 
$\partial_4$ of $F$ are represented by the matrices
\[
\left(\begin{array}{l l l} x & y & z \end{array}\right) 
\]
and
\[
\left(\begin{array}{c c c c c c c c c c c c c c c} z&y&x&0&0&0&0&0&0&0&0&0&0&0&0 \\
0&z&0&-y&-x&0&0&0&0&0&0&0&0&0&0 \\
0&0&z&0&y&-x&0&0&0&0&0&0&0&0&0 \\
0&0&0&z&0&0&-y&x&0&0&0&0&0&0&0 \\
0&0&0&0&0&0&z&0&0&0&y&x&0&0&0 \\
0&0&0&0&z&0&0&-y&-x&0&0&0&0&0&0 \\
0&0&0&0&0&0&0&z&0&0&0&y&-x&0&0 \\
0&0&0&0&0&z&0&0&-y&-x&0&0&0&0&0 \\
0&0&0&0&0&0&0&0&z&0&0&0&y&-x&0 \\
0&0&0&0&0&0&0&0&0&z&0&0&0&y&x \end{array}\right)
\]
respectively. Thus, according to Theorem \ref{thm:realize}, the module $M$ having support variety $\Z(x_1x_2-x_3^2)$ is the cokernel of the map represented by the following $11\times 18$ matrix
\[ \arraycolsep=4.5pt\def\arraystretch{1}
\left( \begin{array}{c c c c c c c c c c c c c c c c c c} 
z&y&x&0&0&0&0&0&0&0&0&0&0&0&0&0&0&0 \\
0&z&0&-y&-x&0&0&0&0&0&0&0&0&0&0&0&0&0 \\
0&0&z&0&y&-x&0&0&0&0&0&0&0&0&0&0&0&0 \\
0&0&0&z&0&0&-y&x&0&0&0&0&0&0&0&0&0&0 \\
0&0&0&0&0&0&z&0&0&0&y&x&0&0&0&0&0&0 \\
0&0&0&0&z&0&0&-y&-x&0&0&0&0&0&0&0&0&0 \\
0&0&0&0&0&0&0&z&0&0&0&y&-x&0&0&0&0&0 \\
0&0&0&0&0&z&0&0&-y&-x&0&0&0&0&0&0&0&0 \\
0&0&0&0&0&0&0&0&z&0&0&0&y&-x&0&0&0&0 \\
0&0&0&0&0&0&0&0&0&z&0&0&0&y&x&0&0&0 \\
-1&0&0&0&0&0&0&0&0&0&0&0&1&0&0&x&y&z \end{array} \right)
\]
\end{example}

\section{Appendix}\label{sec:appendix}

We establish core facts that were used in the main body of the paper.

\subsection{Chain maps and short exact sequences}\label{appendix:ses}
Let $A$ be a ring and $M$ an $A$-module.  Suppose that $F$ is a projective resolution 
of $M$ over $A$.  Choose $[p]\in \Ext_A^d(M,M)$.  Then $[p]$ is a homotopy equivalence class of a chain map
$p: \Sigma^{-d}F \to F$.  Consider the augmented mapping cone given by:
\[
\augcone(p):  \cdots \to F_{d+1}\oplus F_2 \xrightarrow{\left(\begin{array}{cc} -\partial_{d+1} & 0\\
p & \partial_2 \end{array}\right)} F_d\oplus F_1 
\xrightarrow{\left(\begin{array}{cc} -\partial_{d} & 0\\
p & \partial_1 \end{array}\right)} F_{d-1}\oplus F_0
\]
We have the obvious short exact sequence of complexes 
\begin{equation}\label{ses}
0\to F \to \augcone(p) \to \Sigma^{-d+1}F \to 0
\end{equation}
which gives rise to a long exact sequence of homology, which, by taking into account exactness, is
\[
0 \to M \to K_p \to \Omega^{d-1}_A(M) \to 0
\]
where $K_p$ is the cokernel of the map $\left(\begin{array}{cc} -\partial_{d} & 0\\
p & \partial_1 \end{array}\right)$.

Now applying $\Hom_A(-,N)$ to (\ref{ses}) we obtain the long exact sequence 
\[
\cdots \to \Ext_A^{i-d}(M,N) \xrightarrow{\xi_N(p)\cdot} \Ext_A^i(M,N) \to \Ext_A^i(K_p,N) \to \Ext_A^{i-d+1}(M,N) \to \cdots
\]
where $\xi_N(p)\cdot$ is left multiplication by the element $\xi_N(p)\in\Ext_A^d(N,N)$.

\subsection{Syzygies}\label{appendix:syzygies} In this section we show that in some cases Ext is independent of which ring syzygies of a module are taken over.

\begin{proposition}
Suppose that $A$ is a local ring, $I$ a proper ideal of $A$ of finite projective dimension and $B=A/I$.
Let $M$ and $N$ be $B$-modules.  Then 
\[
\Ext_A^i(\Omega^n_B(M),N)\cong\Ext_A^i(\Omega^n_A(M),N)
\]
for all $n\ge 0$ and all $i>\pd_AB$.
\end{proposition}

\begin{proof} We induct on $n$.  The case $n=0$ is trivial.  Suppose the result is true for $n-1\ge 0$.  The short exact sequence $0\to \Omega^n_B(M) \to F \to \Omega^{n-1}_B(M) \to 0$, where $F$ is a free $B$-module, gives rise to the long exact sequence
\begin{align*}
\cdots\to\Ext_A^i(F,N)\to\Ext_A^i(\Omega^n_B(M),N)\to\Ext_A^{i+1}(\Omega^{n-1}_B&(M),N)\\
& \to\Ext_A^{i+1}(F,N) \to \cdots
\end{align*}
Since $\Ext_A^i(F,N)=0$ for all $i>\pd_AB$, we see that 
\[
\Ext_A^i(\Omega^n_B(M),N)\cong\Ext_A^{i+1}(\Omega^{n-1}_B(M),N)
\] 
for all $i>\pd_AB$.  This fact, induction and dimension shifting thus give 
\begin{align*}
\Ext_A^i(\Omega^n_B(M),N)&\cong\Ext_A^{i+1}(\Omega^{n-1}_B(M),N)\\
                         &\cong\Ext_A^{i+1}(\Omega^{n-1}_A(M),N)\\
                         &\cong\Ext_A^i(\Omega^n_A(M),N)
\end{align*}
for all $i>\pd_AB$.
\end{proof}

\subsection{Regular Sequences}\label{apprendix:regseq}

We recall some basic facts regarding regular sequences in Cohen-Macaulay local rings.  Let $(A,\n,k)$ be a Cohen-Macaulay local ring, $f_1,\dots,f_c$ an $A$-regular sequence contained in $\n$, and 
$I=(f_1,\dots,f_c)$ the ideal they generate. 
\begin{enumerate}
\item If $g_1,\dots,g_r$ is a part of a minimal generating set for $I$, then $g_1,\dots,g_r$ form an $A$-regular sequence \cite[Theorem 2.1.2(c)]{BrunsHerzog}.
\item A consequence of Nakayama's Lemma is the following: $g_1,\dots,g_r$ are part of a minimal generating set
for $I$ if and only if $g_1+\n I,\dots,g_r+\n I$ are linearly independent in $I/\n I$.
\end{enumerate}


\subsection{Cohomology operators}\label{cohomology operators}

Let $R$ be a complete intersection of codimension $c$, that is,
\[
R = Q/ (f_1, \dots, f_c)
\]
where $(Q, \n, k)$ is a regular local ring, and $f_1, \dots, f_c$ is a regular sequence contained in $\n^2$. If $M$ is any $R$-module, then the above presentation of $R$ induces, for each $i \in \{ 1, \dots, c \}$, a degree $-2$ endomorphism of any free resolution of $M$, see \cite{Eisenbud} for details. As shown in \cite[Section 1]{Avramov} and \cite[Section 1]{AvramovBuchweitz}, these endomorphisms are images of generators of a certain polynomial ring over $R$. Moreover, this polynomial ring is ``cohomologically central", and all cohomology modules 
$$\Ext_R^*(M,N) \stackrel{\text{def}}{=} \oplus_{n=0}^{\infty} \Ext_R^n(M,N)$$
over $R$ are finitely generated over it. To be precise, there is a polynomial ring $R[ \bm{\chi} ] = R[ \chi_1, \dots, \chi_c]$, in $c$ commuting \emph{cohomology operators} $\chi_1, \dots, \chi_c$, satisfying the following:
\begin{enumerate}
\item The degree of each $\chi_i$ is $2$.
\item For every $R$-module $M$, there is a homomorphism
$$R[ \bm{\chi}] \xrightarrow{\varphi_M} \Ext_R^*(M,M)$$
of graded $R$-algebras.
\item For any pair $(M,N)$ of $R$-modules, the $R[ \bm{\chi} ]$-module structures on the cohomology $\Ext_R^*(M,N)$ via $\varphi_N$ (left module structure) and $\varphi_M$ (right module structure) coincide. That is, if $\eta \in R[ \bm{\chi} ]$ and $\theta \in \Ext_R^*(M,N)$ are homogeneous elements, then 
$\varphi_N(\eta) \cdot \theta = \theta \cdot \varphi_M(\eta)$.
\item For any pair $(M,N)$ of finitely generated $R$-modules, the $R[ \bm{\chi} ]$-module $\Ext_R^*(M,N)$ is finitely generated.
\end{enumerate}
Motivated by the classical theory of cohomological support varieties over group algebras, Avramov and Buchweitz introduced support varieties for complete intersections in \cite{AvramovBuchweitz}. Namely, the support variety of a pair of modules is the variety defined by the annihilator of the cohomology in $k [ \bm{\chi} ]$. By the last property above, for all finitely generated $R$-modules $M$ and $N$ the graded $k$-vector space $k \otimes_R \Ext_R^*(M,N)$ is a finitely generated module over $k \otimes_R R[ \bm{\chi} ]$. The latter is naturally identified with the polynomial ring $k [ \bm{\chi} ]$.

\begin{definition}[Avramov-Buchweitz]
Let $(Q, \n, k)$ be a regular local ring and $f_1, \dots, f_c$ a regular sequence contained in $\n^2$. Denote by $R$ the complete intersection ring $R = Q/ (f_1, \dots, f_c)$, and for $R$-modules $M$ and $N$, denote the $k [ \bm{\chi} ]$-module $k \otimes_R \Ext_R^*(M,N)$ by $E(M,N)$.

(1) For a pair $(M,N)$ of finitely generated $R$-modules, the support variety $\V_R(M,N)$ is defined as
$$\V_R(M,N) \stackrel{\text{def}}{=} \{ \bm{\alpha} \in \overline{k}^c \mid \phi ( \bm{\alpha} ) =0 \text{ for all } \phi \in \Ann_{k [ \bm{\chi} ]} E(M,N) \},$$
where $\overline{k}$ is the algebraic closure of $k$.

(2) For a single finitely generated $R$-module $M$, the support variety $\V_R(M)$ is defined as
$$\V_R(M) \stackrel{\text{def}}{=} \V_R(M,M).$$
\end{definition}

We mention some properties of these cohomological support varieties. Firstly, since $\Ext_R^*(M,N)$ is a graded module over the ring $R [ \bm{\chi} ]$ of cohomology operators, the $k [ \bm{\chi} ]$-module $E(M,N)$ is also graded. Its annihilator $\Ann_{k [ \bm{\chi} ]} E(M,N)$ is therefore a homogeneous ideal, hence the support variety $\V_R(M,N)$ contains the origin and is a \emph{cone} in affine $c$-space $\mathbb{A}^c ( \overline{k} )$. Secondly, by \cite[Theorem 5.6(9)]{AvramovBuchweitz}, for every finitely generated $R$-module $M$ there are equalities
$$\V_R(M,M) = \V_R(M,k) = \V_R(M,k).$$
Consequently, the support variety $\V_R(M)$ can be defined as any of these three varieties. Finally, the main reason why this theory is so powerful, as illustrated by the numerous results in \cite{AvramovBuchweitz}, is the fact that cohomological finiteness holds. In other words, the fact that for all finitely generated $R$-modules $M$ and $N$, the $R[ \bm{\chi} ]$-module $\Ext_R^*(M,N)$ is finitely generated. This is precisely the reason why the support varieties encode many of the homological properties of the modules involved. For example, the dimension of the variety of a module equals the complexity of the module. In particular, a module has trivial variety if and only if it has finite projective dimension (or, equivalently, finite injective dimension, since complete intersections are Gorenstein).

\subsection{Algebraic closure}\label{appendix:algebraically closed} Let $R$ be a complete intersection, that is, $R=Q/I$ for $(Q,\n,k)$ a regular local ring, and $I$ an ideal generated by a $Q$-regular sequence contained in $\n^2$.  Let $M$ and $N$ be finitely generated $R$-modules.  If $k$ is not algebraically closed, then there exists a faithfully flat extension $Q \to Q'$ such that the residue field $k'$ of $Q'$ is algebraically closed (see \cite[App., Th\'{e}or\`{e}m 1, Corollarie]{Bourbaki}).  In this case $R'=Q'/IQ'$ is a complete intersection with algebraically closed residue field, and for any ideal $J$ of $Q$, the induced map
$Q/J\to Q'/JQ'$ is a faithfully flat.  We define the support variety $\V_R(M,N)$ of the pair $(M,N)$ over $R$ as the support variety of the pair $(M\otimes_R R',N\otimes_RR')$ over $R'$.  Since $\Ext^n_{Q/J}(M,N)$ vanishes if and only if
so does $\Ext^n_{Q'/JQ'}(M\otimes_RR',N\otimes_RR')$ for every $n$, and for any ideal $J$ of $Q$, the results of the paper go through without assuming that $k$ is algebraically closed.

\subsection{Hypersurfaces}\label{appendix:hypersurface}  In this subsection we prove the following.

\begin{proposition} Suppose that $R$ is a hypersurface ring, that is, a quotient of a regular local ring 
$(Q,\n,k)$ by a principal ideal generated by a nonzero element $f\in \n$.  Let $M$ be a finitely generated $R$-module.  If $\Ext_R^i(M,k)\ne 0$ for infinitely many $i$, then right multiplication on 
$\Ext_R(M,k)$ by $\xi_k(\chi)$ is an isomorphism in all sufficiently high degrees.
\end{proposition}  

\begin{proof}  The hypothesis that $\Ext_R^i(M,k)\ne 0$ for infinitely many $i$ implies that $M$ has infinite projective dimension over $R$.  In this case $R$ must be a singular hypersurface (i.e. $f\in\n^2$) and it is well-known that $k$ has an infinite minimal free resolution over $R$, which is eventually 2-periodic 
\cite{Eisenbud}.  Since $R$ is a singular hypersurface, the ring of cohomology operators is a polynomial ring in a single cohomology operator variable $\chi$.  The map
\[
\varphi_k:R[\chi]\to\Ext_R^*(k,k)
\]
sends $\chi$ to the homotopy equivalence class of the periodicity endomorphism of the minimal free resolution of $k$.  This shows that left multiplication by $\varphi_k(\chi)$ on $\Ext_R^*(M,k)$ is an isomorphism in all sufficiently high degrees.
\end{proof}


\begin{thebibliography}{HuW}
\bibitem[Avr]{Avramov}L.L.\ Avramov, \emph{Modules of finite virtual projective dimension}, Invent.\ Math.\ 96 (1989), no.\ 1, 71--101.
\bibitem[AvB]{AvramovBuchweitz}L.L.\ Avramov, R.-O.\ Buchweitz, \emph{Support varieties and cohomology over complete intersections}, Invent.\ Math.\ 142 (2000), no.\ 2, 285--318.
\bibitem[AvI]{AvramovIyengar}L.L.\ Avramov, S.\ Iyengar, \emph{Modules with prescribed cohomological support}, Ill.\ J.\ Math.\ 51 (2007), no.\ 1, 1--20.
\bibitem[Ber]{Bergh} P.\ A.\ Bergh \emph{On support varieties for modules over complete intersections}, Proc. Amer. Math. Soc. 135 (2007), no. 12, 3795--3803.
\bibitem[Bou]{Bourbaki}N.\ Bourbaki, \emph{Alg{\`e}bre commutative. IX: Anneaux locaux r\'eguliers complets}, Masson, Paris, 1983. 
\bibitem[BrH]{BrunsHerzog}W.\ Bruns, J.\ Herzog \emph{Cohen-Macaulay Rings}
Cambridge studies in advanced mathematics 39.
\bibitem[Eis]{Eisenbud}D.\ Eisenbud, \emph{Homological algebra on a complete intersection, with an application to group representations}, Trans.\ Amer.\ Math.\ Soc.\ 260 (1980), no.\ 1, 35--64.
\bibitem[HuJ]{HunekeJorgensen} C.\ Huneke, D.\ A.\ Jorgensen, \emph{Symmetry in the vanishing of Ext over Gorenstein rings} Math.\ Scand. 93 (2003), no. 2, 161--184.
\bibitem[HJK]{HunekeJorgensenKatz} C.\ Huneke, D.\ A.\ Jorgensen, D.\ Katz, \emph{Fitting ideals and finite projective dimension}, Math. Proc.\ Camb.\ Phil.\ Soc. (2005),\ 138, 41--54.  
\bibitem[HuW]{HunekeWiegand} C.\ Huneke, R. Wiegand, \emph{Tensor products of modules, rigidity and local cohomology} Nath.\ Scand.\ (1997), 161--183.
\bibitem[Jo1]{Jorgensen1}D.A.\ Jorgensen, \emph{Support sets of pairs of modules}, Pacific J.\ Math.\ 207 (2002), no.\ 2, 393--409.
\bibitem[Jo2]{Jorgensen2}D. A. Jorgensen, \emph{A generalization of the Auslander-Buchsbaum formula}, J. Pure Appl. Algebra
144 (1999), 145--155.
\bibitem[OPo]{O'CarrollPopescu}L.\ O'Carroll, D.\ Popescu, \emph{Splitting syzygies}. J. Algebra 228 (2000), no. 2, 682--709.
\end{thebibliography}
\end{document}